\documentclass[a4paper,12pt]{article}

\usepackage{color,amsmath,amssymb,bm,tikz}

\unitlength\textwidth
\divide\unitlength by 200\relax

\bmdefine{\sss}{s}
\bmdefine{\vvv}{v}

\DeclareMathAlphabet{\mathscr}{U}{rsfs}{m}{n}

\newcommand{\msCCC}{\mathscr{C}}
\newcommand{\msOOO}{\mathscr{O}}
\newcommand{\msPPP}{\mathscr{P}}

\newcommand{\msIII}{\mathscr{I}}

\newcommand{\msXXX}{\mathscr{X}}

\newcommand{\NNN}{\mathbb{N}}
\newcommand{\ZZZ}{\mathbb{Z}}
\newcommand{\QQQ}{\mathbb{Q}}
\newcommand{\RRR}{\mathbb{R}}
\newcommand{\KKK}{\mathbb{K}}

\newcommand{\aaaa}{\mathfrak{a}}
\newcommand{\mmmm}{\mathfrak{m}}
\newcommand{\nnnn}{\mathfrak{n}}
\newcommand{\pppp}{\mathfrak{p}}
\newcommand{\qqqq}{\mathfrak{q}}

\newcommand{\RRRRR}{{\mathcal R}}

\newcommand{\TTTTT}{{\mathcal T}}

\newcommand{\tUUUUU}{{{t}\mathcal U}}
\newcommand{\qUUUUU}{{{q}\mathcal U}}

\newcommand{\TTTTTn}{{\mathcal T}^{(n)}}

\newcommand{\tUUUUUn}{{{t}\mathcal U}^{(n)}}
\newcommand{\qUUUUUn}{{{q}\mathcal U}^{(n)}}

\newcommand{\covered}{\mathrel{<\!\!\!\cdot}}
\newcommand{\define}{\mathrel{:=}}

\newcommand{\Cl}{{\mathrm{Cl}}}
\newcommand{\cl}{{\mathrm{cl}}}
\newcommand{\gor}{Gorenstein}
\newcommand{\cm}{Cohen-Macaulay}
\newcommand{\bbm}{Buchsbaum}

\newcommand{\nzd}{non-zero-divizor}
\newcommand{\gr}{\mathrm{Gr}}

\newcommand{\rank}{\mathrm{rank}}

\newcommand{\cdeg}{{\mathrm{cdeg}}}

\newcommand{\ann}{{\mathrm{ann}}}
\newcommand{\ass}{{\mathrm{Ass}}}
\newcommand{\assh}{{\mathrm{Assh}}}
\newcommand{\depth}{{\mathrm{depth}}}
\newcommand{\trace}{{\mathrm{tr}}}

\newcommand{\Div}{{\mathrm{Div}}}
\renewcommand{\hom}{{\mathrm{Hom}}}
\newcommand{\ext}{{\mathrm{Ext}}}
\newcommand{\coker}{{\mathrm{Coker}}}
\newcommand{\spec}{{\mathrm{Spec}}}

\newcommand{\conv}{\mathrm{conv}}
\newcommand{\stab}{\mathrm{STAB}}

\newcommand{\xip}{\xi^{+}}

\newcommand{\ekp}{E_\KKK[\msPPP]}

\newcommand{\ekop}{{E_\KKK[\msOOO(P)]}}
\newcommand{\eksg}{{E_\KKK[\stab(G)]}}

\newtheorem{thm}{Theorem}[section]
\newtheorem{fact}[thm]{Fact}
\newtheorem{example}[thm]{Example}
\newtheorem{lemma}[thm]{Lemma}
\newtheorem{cor}[thm]{Corollary}
\newtheorem{definition}[thm]{Definition}
\newtheorem{prop}[thm]{Proposition}
\newtheorem{remark}[thm]{Remark}

\newtheorem{prob}[thm]{Problem}

\newtheorem{claim}{Claim}[thm]

\newcommand{\bigzerou}{\smash{\lower1.7ex\hbox{\bg 0}}}

\newcommand{\bigastu}{\smash{\lower1.7ex\hbox{\bg *}}}

\numberwithin{equation}{section}

\newcommand{\mylabel}[1]{{\label{#1}\tt [#1]}}
\let\mylabel=\label

\title{%
Radical property of the traces of the canonical modules of \cm\ rings}  
\author{Mitsuhiro Miyazaki%
\\
\normalsize
Osaka Central Advanced Mathematical Institute,\\
\normalsize
Osaka Metropolitan University,\\
\normalsize
3-3-138 Sugimoto, Sumiyoshi-ku Osaka 558-8585, Japan%
}
\date{
\tt e-mail:mmyzk7@gmail.com}

\begin{document}

\maketitle

\sloppy

\begin{abstract}
In this paper, we define a new concept of Noetherian commutative rings which stands between
\gor\ and \cm\ properties.
We show that this new property keep hold under common operations of commutative rings
such as localization, polynomial extension and under mild assumptions, flat extension, tensor product,
Segre product and so on.
We show that for Schubert cycles, the Ehrhart rings of cycle graphs and perfect graphs, this new
concept is close to \gor\ property.
\\
MSC: 13H10, 14M15, 05E40, 05C25, 52B20\\
 {\bf Keywords}: \gor, \cm, trace ideal, canonical module
\end{abstract}



\section{Introduction}

There is a hierarchy of commutative Noetherian local rings.
\begin{eqnarray*}
&&\mbox{regular}\Rightarrow\mbox{complete intersection}\\
&&\Rightarrow
\mbox{\gor}\Rightarrow\mbox{\cm}\Rightarrow\mbox{\bbm}.
\end{eqnarray*}
\cm\ rings are first defined as a Noetherian rings which satisfy unmixedness theorem.
Macaulay showed that unmixedness theorem holds for polynomial rings over a field and Cohen
showed that unmixedness theorem holds for regular local rings.
In old days, \cm\ rings are sometimes called semi-regular rings.
Further, there are many situations that rings under consideration is \cm.
Therefore \cm\ rings are platform of many theories.

On the other hand, \gor\ ring is defined by Bass \cite{bas}.
A \gor\  local ring is by definition a commutative Noetherian local ring whose self injective dimension is finite.
A \gor\ ring is a \cm\ ring whose parameter ideal is irreducible
and has very beautiful properties especially concerning symmetry of related objects such as syzygies and dualities.

In pursuing the study of \gor\ and \cm\ rings, many researchers felt that there is a rather large gap between
\gor\ and \cm\ properties.
Thus, there were attempts to define notions between \gor\ and \cm\ properties and fill this gap.
The first one is the level property defined by Stanley \cite{sta1}.
However, level property can be defined only for semi-standard graded rings over a field.

After that, almost \gor\ property \cite{bf, gmp, gtt}
and nearly \gor\ property \cite{hhs} are defined.
However, there are few rings that are non-\gor\ but almost or nearly \gor.

In this paper, we define a new notion, that we call canonical trace radical (CTR for short) property,
which stands between \gor\ and \cm\ properties.
We define a \cm\ local ring is CTR if it admits a canonical module and its trace ideal is a radical ideal.
We show that, under mild assumptions, CTR property is retained under the frequently used operations
of Noetherian rings, such as localization, flat extension, tensor product, Segre product and some others.
We also show that in some combinatorial rings, combinatorial property corresponding to CTR property is
close to combinatorial property corresponding to \gor\ property.

Recently, Esentepe \cite{ese} treated the radical property of the trace ideal of the canonical module of a
\cm\ ring in relation to Auslander-Reiten conjecture.

This paper is organized as follows.
In \S2, we establish notation and terminology used in this paper and recall some basic facts,
especially the trace of a module.
In \S3, we define the canonical trace radical (CTR for short) property and study under what kind of
ring operations, with sometimes additional assumptions, CTR property is retained:
localization, flat extension, polynomial extension, completion, division by a regular sequence, tensor product
and Segre product.

In \S\S4 and 5, we state criteria of CTR property for certain classes of rings which motivated us to define
CTR property.
In \S4, we deal with Schubert cycles, i.e. the homogeneous coordinate rings of the Schubert
subvarieties of Grassmannians:
let $R$ be a Schubert cycle.
Then by \cite[\S8]{bv}, it is known that $R$ is a \cm\ normal domain and there are height 1 prime ideals
$P_0$, $P_1$, \ldots, $P_t$ such that the divisor class group $\Cl(R)$ is generated by 
$\cl(P_0)$, $\cl(P_1)$, \ldots, $\cl(P_t)$ and $\sum_{i=0}^t\cl(P_i)=0$ is the only relation between
them.
Let $\sum_{i=0}^t\kappa_i\cl(P_i)$ be the canonical class of $\Cl(R)$,
$\kappa=\max\{\kappa_i : 0\leq i\leq t\}$ and
$\kappa'=\min\{\kappa_i : 0\leq i\leq t\}$.
Then $\kappa-\kappa'$ is independent of the representation of the canonical class above.
It is known that $R$ is \gor\ if and only if $\kappa-\kappa'=0$.
We show that $R$ is CTR if and only if $\kappa-\kappa'\leq 1$.
See Theorem \ref{thm:schubert}.

In \S5, we deal with the Ehrhart ring of the stable set polytope of a cycle graph:
let $R$ be such a ring.
It is known that $R$ is \gor\ if and only if the length $n$ of the cycle is even or less than 7.
We show that $R$ is CTR if and only if  $n$ is even or less than 9.
See Theorem \ref{thm:cycle graph}.

Finally in \S6, we state a necessary condition that the Ehrhart ring $R$ of the stable set polytope of 
a perfect graph $G$ is CTR:
set $k=\max\{|K|:$ $K$ is a maximal clique in $G\}$ and
$k'=\min\{|K|:$ $K$ is a maximal clique in $G\}$.
It is known that $R$ is \gor\ if and only if $k-k'=0$.
We show that if $R$ is CTR, then $k-k'\leq 1$.
See Proposition \ref{prop:perfect graph}.


\section{Preliminaries}

\mylabel{sec:pre}

In this section, we establish notation and terminology used in this paper.
For unexplained term of commutative algebra, we consult \cite{mat} and \cite{bh}.

All rings and algebras are assumed to be commutative with identity element and
Noetherian.
We denote the set of nonnegative integers, 
the set of integers, 
the set of rational numbers and 
the set of real numbers
by $\NNN$, $\ZZZ$, $\QQQ$ and $\RRR$ respectively.

For a set $X$, we denote by $|X|$ the cardinality of $X$.
For sets $X$ and $Y$, we denote by $X\setminus Y$ the set $\{x\in X: x\not\in Y\}$.
For nonempty sets $X$ and $Y$, we denote the set of maps from $X$ to $Y$ by $Y^X$.
If $X$ is a finite set, we identify $\RRR^X$ with the Euclidean space 
$\RRR^{|X|}$.
For $f$, $f_1$, $f_2\in\RRR^X$ and $a\in \RRR$,
we define
maps $f_1\pm f_2$ and $af$ 
by
$(f_1\pm f_2)(x)=f_1(x)\pm f_2(x)$ and
$(af)(x)=a(f(x))$
for $x\in X$.
Let $A$ be a subset of $X$.
We define the characteristic function $\chi_A\in\RRR^X$ of $A$ by
$\chi_A(x)=1$ for $x\in A$ and $\chi_A(x)=0$ for $x\in X\setminus A$.
For a nonempty subset $\msXXX$ of $\RRR^X$, we denote by $\conv\msXXX$
the convex hull 
of $\msXXX$.

Next we fix notation about Ehrhart rings.
Let $\KKK$ be a field, $X$ a finite set  
and $\msPPP$ a rational convex polytope in $\RRR^X$, i.e. 
a convex polytope whose vertices are contained in $\QQQ^X$.
Let $-\infty$ be a new element
with $-\infty\not\in X$ and
set $X^-\define X\cup\{-\infty\}$.
Also let $\{T_x\}_{x\in X^-}$ be
a family of indeterminates indexed by $X^-$.
For $f\in\ZZZ^{X^-}$, 
we denote the Laurent monomial 
$\prod_{x\in X^-}T_x^{f(x)}$ by $T^f$.
We set $\deg T_x=0$ for $x\in X$ and $\deg T_{-\infty}=1$.
Then the Ehrhart ring of $\msPPP$ over a field $\KKK$ is the $\NNN$-graded subring
$$
\KKK[T^f : f\in \ZZZ^{X^-}, f(-\infty)>0, \frac{1}{f(-\infty)}f|_X\in\msPPP]
$$
of the Laurent polynomial ring $\KKK[T_x^{\pm1} : x\in X^-]$,
where $f|_X$ is the restriction of $f$ to $X$.
We denote the Ehrhart ring of $\msPPP$ over $\KKK$ by $\ekp$.
If $X$ is a poset, we define the order on $X^-$ by $-\infty<x$ for any $x\in X$.

Let $R$ be a ring.
For $\pppp\in\spec(R)$, we denote by $\kappa(\pppp)$ the quotient field
$R_\pppp/\pppp R_\pppp$ of $R/\pppp$.
For an ideal $I$ of $R$, we denote by $\min(I)$ the set of minimal over primes of $I$.
For a ring $R$ and a matrix $M$ with entries in $R$, we denote by $I_t(M)$ the ideal
of $R$ generated by $t$-minors of $M$.

An $\NNN$-graded ring $R=\bigoplus_{n\in\NNN}R_n$ is said to be an $\NNN$-graded
$\KKK$-algebra if $R_0=\KKK$.
An $\NNN$-graded $\KKK$-algebra $R=\bigoplus_{n\in\NNN}R_n$ is said to be standard 
graded if $R=\KKK[R_1]$.
For an $\NNN$-graded ring $R$, a greded $R$-module $M$ and $m\in\ZZZ$, we denote
by $M_{\geq m}$ the graded $R$-submodule $\bigoplus_{n\geq m}M_n$ of $M$.
When $R$ is a \cm\ local ring with a canonical module or an $\NNN$-graded algebra
over a field, we denote by $\omega_R$ the (graded) canonical module of $R$.
If $R$ is an $\NNN$-graded algebra and $\omega_R$ is the canonical module of $R$,
$-\min\{m:(\omega_R)_m\neq 0\}$ is called the $a$-invariant of $R$ and denoted
by $a(R)$.
See \cite{gw}.
For $\NNN$-graded $\KKK$-algebras $R^{(1)}$, \ldots,  $R^{(m)}$, we denote by
$R^{(1)}\#\cdots\#R^{(m)}$ the Segre product 
$\bigoplus_{n\in\NNN}R^{(1)}_n\otimes\cdots\otimes R^{(m)}_n$ of $R^{(1)}$, \ldots, $R^{(m)}$.

Let $R=\bigoplus_{n\in\NNN}R_n$ be a standard graded $\KKK$-algebra,
where $\KKK$ is a field.
We say that $R$ is level if the graded canonical module $\omega_R$ is generated in one degree.
If $\hom_R(\omega_R, R)$ (whice is a graded module. See \cite{gw}.) is generated in one degree,
we say that $R$ is anticanonical level.
Level and anticanonical level properties are independent,
see \cite{pag, mfiber, mchain}.
We denote $\hom_R(\omega_R, R)$ by $\omega_R^{-1}$.

Now we recall the following.

\begin{definition}
\rm
Let $R$ be a ring and $M$ an $R$-module.
We define the trace of $M$ denoted by $\trace_R(M)$ by
$$
\trace_R(M)\define\sum_{\varphi\in\hom_R(M,R)}\varphi(M).
$$
If $R$ is clear from context, we omit the subscript $R$ and denote $\trace(M)$.
\end{definition}
It follows from the definition, the following.

\begin{lemma}
\mylabel{lem:tr basic}
Let $R$ be a ring and $M$ an $R$-module.
Then $\trace_R(M)$ is the image of the canonical map
$$
\hom_R(M, R)\otimes M\to R, \quad f\otimes m\mapsto f(m).
$$
In particular, if $M$ is a finitely generated $R$-module and $S$ is a flat $R$-algebra,
then $\trace_S(M\otimes S)=\trace_R(M)S$ and therefore for $\pppp\in\spec(R)$,
$\trace_{R_\pppp}(M_\pppp)=\trace_R(M)_\pppp$.
\end{lemma}
Flat extension part of this lemma is shown in \cite[Lemma 1.5 (iii)]{hhs}.
Further, if $I$ is an ideal of $R$ and contains an $R$-regular element $b$, then
$\hom_R(I, R)\cong \{x\in (1/b)R : xI\subset R\}\subset Q(R)$,
where $Q(R)$ is the total quotient ring of $R$,
 and therefore, 
$\trace_R(I)=I\{x\in (1/b)R : xI\subset R\}$.
See the proof of \cite[Lemma 1.1]{hhs}.

Let $R$ be a normal domain.
For divisorial fractionary ideal $I$ of $R$ and $n\in\ZZZ$, we denote by
$I^{(n)}$ the $n$-th power of $I$ in $\Div(R)$.
Note that if $I$ is a height 1 prime ideal of $R$ and $n>0$, then
$I^{(n)}$ coincide with the $n$-the symbolic power of $I$.
Further, 
by the argument in the previous paragraph,
$\trace_R(I)=II^{(-1)}$ for any divisorial fractionary ideal $I$.

Next we state a tool to compute the trace of a canonical module which is a
generalization of \cite[Corollary 3.2]{hhs}.
A homomorphism $\varphi\colon F\to G$ of finitely generated free modules over a
ring can be expressed by a matrix by fixing bases of $F$ and $G$.
Let $M$ be such a matrix and $t$ a positive integer with $t\leq\min\{\rank F, \rank G\}$.
Then the ideal $I_t(M)$ is independent of the choice of bases of $F$ and $G$.
We denote this ideal by $I_t(\varphi)$.

\begin{lemma}
\mylabel{lem:hhs 3.2}
Let $S$ be a \gor\ local ring ($\NNN$-graded ring over a field), $J$ a (homogeneous) ideal of $S$ 
such that $R=S/J$ is a \cm\ ring.
Suppose that there exists a finite (graded) $S$-free resolution
$$
0\to F_h\stackrel{\varphi_h}{\to} \cdots\stackrel{\varphi_1}{\to} F_0\to R\to 0
$$
of $R$ with $h=\dim S-\dim R$.
Let $G$ be a free $R$-module and
$\psi\colon G\to F_h\otimes R$ a (graded) $R$-homomorphism with
$$
G\stackrel{\psi}{\to} F_h\otimes R\stackrel{\varphi_h\otimes 1}{\to}F_{h-1}\otimes R
$$
is exact.
Then $\trace_R(\omega_R)=I_1(\psi)$.
\end{lemma}
\begin{proof}
Since $\omega_R=\ext_S^h(R,\omega_S)=\ext_S^h(R,S)$,
we see that $\omega_R\cong\coker((\varphi_h\otimes 1)^\ast)$.
Therefore, by \cite[Proposition 3.1]{hhs}, we see the result.
\end{proof}


\section{Canonical trace radical rings}

In this section, we define the notion of canonical trace radical rings
(CTR rings for short) and study basic properties of CTR rings.
The reasons that we think CTR property is close to \gor\ property is shown in the following sections.
First we recall the following.

\begin{fact}
\mylabel{fact:gor trace}
Let $R$ be a \cm\ local ring with canonical module $\omega_R$.
Then $R$ is \gor\ if and only if $\trace_R(\omega_R)=R$.
\end{fact}
For the proof, see e.g. \cite[Lemma 2.1]{hhs}.

\begin{definition}
\mylabel{def:ctr}
\rm
Let $R$ be a \cm\ ring.
If for any $\pppp\in\spec(R)$, $R_\pppp$ has a canonical module $\omega_{R_\pppp}$ and
$\trace_{R_\pppp}(\omega_{R_\pppp})$ is a radical ideal,
then we say that $R$ is a canonical trace radical (CTR for short) ring.
\end{definition}
Since the unit ideal is a radical ideal, a \gor\ ring is a CTR ring.
Further, it is evident from the definition that CTR property is kept by localization.

By Lemma \ref{lem:tr basic}, we see the following.

\begin{prop}
\mylabel{prop:local}
 Let $(R,\mmmm)$ be a \cm\ local ring with canonical module.
 Then $R$ is CTR if and only if $\trace_R(\omega_R)$ is a radical ideal.
 \end{prop}
 \begin{proof}
 ``Only if'' part is a direct consequence of the definition.
 We prove the ``if'' part.
 Let $\pppp$ be an arbitrary prime ideal of $R$.
 By Lemma \ref{lem:tr basic}, we see that
 $\trace_{R_\pppp}((\omega_R)_\pppp)=\trace_R(\omega_R)_\pppp$.
 Since $\omega_{R_\pppp}=(\omega_R)_\pppp$ and $\trace_R(\omega_R)_\pppp$ is 
 a radical ideal of $R_\pppp$, we see the result.
 \end{proof}
Next we show a similar fact to the above proposition, which may be regarded as a graded
version of the above proposition.

\begin{prop}
\mylabel{prop:graded}
Let $R$ be an $\NNN$-graded \cm\ ring over a field $\KKK$ and $\omega_R$ the graded
canonical module of $R$.
Then $R$ is CTR if and only if $\trace_R(\omega_R)$ is a radical ideal.
\end{prop}
\begin{proof}
We first prove the ``only if'' part.
Let $\mmmm$ be the irrelevant maximal ideal of $R$.
Then $(\omega_R)_\mmmm$ is the canonical module of $R_\mmmm$.
Therefore, by assumption and Lemma \ref{lem:tr basic},
$\trace_R(\omega_R)_\mmmm=\trace_{R_\mmmm}((\omega_R)_\mmmm)$ is a radical ideal.
Since $\trace_R(\omega_R)$ is a graded ideal, every associated prime of 
$\trace_R(\omega_R)$ is graded and therefore contained in $\mmmm$.
Thus, the radical property of $\trace_R(\omega_R)_\mmmm$ implies the radical property 
of $\trace_R(\omega_R)$.

Now we prove the ``if'' part.
Let $\pppp$ be an arbitrary prime ideal of $R$.
Take a polynomial ring $S$ with weighted degree over $\KKK$ and a graded surjective
$\KKK$-algebra homomorphism $S\to R$.
Let $P$ be the preimage of $\pppp$.
Also take a minimal graded $S$-free resolution
$$
0\to F_h\stackrel{\varphi_h}{\to}\cdots\stackrel{\varphi_1}{\to} F_0\to R\to 0
$$
of $R$, a free $R$-module $G$ and a graded $R$-homomorphism $\psi$ such that
$$
G\stackrel{\psi}{\to} F_h\otimes R \stackrel{\varphi_h\otimes 1}{\to} F_{h-1}\otimes R
$$
is exact.
Then
$$
0\to (F_h)_P\stackrel{(\varphi_h)_P}{\to}\cdots\stackrel{(\varphi_1)_P}{\to} (F_0)_P\to R_\pppp\to 0
$$
is a (not necessarily minimal) $S_P$-free resolution of $R_\pppp$ and
$$
G_\pppp\stackrel{\psi_\pppp}{\to} (F_h)_P\otimes_{S_P} R_\pppp \stackrel{(\varphi_h)_P\otimes 1}{\to} 
(F_{h-1})_P\otimes_{S_P} R_\pppp
$$
is exact.

Since $R$ is \cm, it follows that $\ass R=\assh R$ and therefore
$h=\dim S-\dim R=\dim S_P-\dim R_\pppp$.
Thus, by Lemma \ref{lem:hhs 3.2}, we see that $\trace_{R_\pppp}(\omega_{R_\pppp})=I_1(\psi_\pppp)$.
Since $\trace_R(\omega_R)=I_1(\psi)$ by Lemma \ref{lem:hhs 3.2} and $\trace_R(\omega_R)$
is a radical ideal, we see that
$$
\trace_{R_\pppp}(\omega_{R_\pppp})=I_1(\psi_\pppp)=(I_1(\psi))_\pppp=\trace_R(\omega_R)_\pppp
$$
is a radical ideal.
\end{proof}
By Propositions \ref{prop:local} and \ref{prop:graded}, we see that nearly \gor\ rings are CTR rings.

Next we consider the CTR property under the flat extension.

\begin{prop}
\mylabel{prop:flat ext}
Let $(R,\mmmm)$ be a local ring with canonical module and 
$(R,\mmmm)\to(S,\nnnn)$ be a flat local homomorphism.
Suppose $S/\mmmm S$ is a \gor\ ring.
Then the followings hold.
\begin{enumerate}
\item
\mylabel{item:down}
If $S$ is CTR, then so is $R$.
\item
\mylabel{item:up}
If $R$ is CTR and for any $\pppp\in\min(\trace_R(\omega_R))$, 
$S_\pppp/\pppp S_\pppp=\kappa(\pppp)\otimes_R S$ is a reduced ring, then $S$ is CTR.
\end{enumerate}
\end{prop}
\begin{proof}
First note that $\omega_R\otimes S$ is the canonical module of $S$, see \cite[Theorem 3.3.14]{bh}.
Therefore, by Lemma \ref{lem:tr basic} we see that $\trace_S(\omega_S)=\trace_R(\omega_R)S$.

\ref{item:down}
Since the natural map
$R/\trace_R(\omega_R)\to S/\trace_R(\omega_R)S=S/\trace_S(\omega_S)$
is faithfully flat, $R/\trace_R(\omega_R)$ is isomorphic to a subring of $S/\trace_S(\omega_S)$.
Since $S/\trace_S(\omega_S)$ is reduced by assumption, $R/\trace_R(\omega_R)$ is also reduced.

\ref{item:up}
Let $P$ be an arbitrary associated prime ideal of $\trace_S(\omega_S)$ and set $\pppp=P\cap R$.
Since 
$(R/\trace_R(\omega_R))_\pppp\to (S/\trace_S(\omega_S))_P$ 
is a flat local homomorphism, 
$\depth (R/\trace_R(\omega_R))_\pppp\leq\depth (S/\trace_S(\omega_S))_P=0$
by \cite[Theorem 23.3]{mat}.
Therefore, $\pppp$ is an associated prime ideal of $\trace_R(\omega_R)$.
Since $\trace_R(\omega_R)$ is a radical ideal, we see that $\pppp\in\min(\trace_R(\omega_R))$
and $(R/\trace_R(\omega_R))_\pppp=\kappa(\pppp)$.
Thus, 
since $\kappa(\pppp)\otimes_R S_P$ is a localization of $\kappa(\pppp)\otimes_R S$, 
we see by assumption that
$(S/\trace_S(\omega_S))_P=(S/\trace_R(\omega_R)S)_P=(R/\trace_R(\omega_R))_\pppp\otimes_R S_P
=\kappa(\pppp)\otimes_R S_P$
is reduced.
Since $P$ is an arbitrary associated prime ideal of $\trace_S(\omega_S)$, we see that
$\trace_S(\omega_S)$ is a radical ideal.
\end{proof}
\begin{cor}
Let $R$ be a ring and $S=R[X_1, \ldots, X_n]$ a polynomial ring over $R$.
Then $S$ is CTR if and only if so is $R$.
\end{cor}
\begin{proof}
First assume that $S$ is CTR.
Let $\pppp$ be an arbitrary prime ideal of $R$.
Set $P=\pppp S$.
Then $P$ is a prime ideal of $S$ and $R_\pppp=S_P/(X_1, \ldots, X_n)S_P$.
Since $S_P$ admits a canonical module, $S_P$ is a homomorphic image of a \gor\ ring.
Therefore, $R_\pppp$ is also a homomorphic image of a \gor\ ring.
Thus, $R_\pppp$ admits a canonical module.

Since $S_P/\pppp S_P=S_P/PS_P$ is a field, 
therefore 
is \gor\ and $R_\pppp\to S_P$ is a flat local homomorphism, we see by 
Proposition \ref{prop:flat ext} that $R_\pppp$ is CTR.
Since $\pppp$ is an arbitrary prime ideal of $R$, we see that $R$ is CTR.

Next we assume that that $R$ is CTR.
Let $P$ be an arbitrary prime ideal of $S$ and set $\pppp=P\cap R$.
Then $S_P/\pppp S_P=\kappa(\pppp)\otimes_R S_P$ is a localization of $\kappa(\pppp)[X_1, \ldots, X_n]$ and therefore \gor.
Further, for any $\qqqq\in \min(\trace_{R_\pppp}(\omega_{R_\pppp}))$, $\kappa(\qqqq)\otimes_{R_\pppp}S_P$ is a localization of
$\kappa(\qqqq)\otimes_{R_\pppp}S=\kappa(\qqqq)\otimes_R S=
\kappa(\qqqq)[X_1,\ldots, X_n]$ and therefore is reduced.
Thus, by Proposition \ref{prop:flat ext}, we see that $S_P$ is CTR.
Since $P$ is an arbitrary prime ideal of $S$, we see that $S$ is CTR.
\end{proof}
Next consider the CTR property under completion.

\begin{prop}
Let $(R,\mmmm)$ be a local ring with canonical module and $\widehat R$ the completion of
$R$ with respect to $\mmmm$.
Then the followings hold.
\begin{enumerate}
\item
\mylabel{item:downc}
If $\widehat R$ is CTR, then so is $R$.
\item
\mylabel{item:upc}
If $R$ is a Nagata ring (pseudo-geometric ring in Nagata's terminology) and CTR, then
$\widehat R$ is also CTR.
\end{enumerate}
\end{prop}
\begin{proof}
Since $\widehat R/\mmmm \widehat R=R/\mmmm$ is a field, \ref{item:downc} follows from
Proposition \ref{prop:flat ext}.
For \ref{item:upc}, note that for any $\pppp\in\spec(R)$, $R/\pppp$ is analytically unramified,
i.e. $\widehat{(R/\pppp)}$ is reduced.
See \cite[Theorem 36.4]{nag}.
Therefore, $(R/\trace_R(\omega_R))\widehat{\ }$ is reduced, since $\trace_R(\omega_R)$ is a radical ideal.
Since $\widehat R/\trace_{\widehat R}(\omega_{\widehat R})=\widehat R/\trace_R(\omega_R)\widehat R
=(R/\trace_R(\omega_R))\widehat{\ }$, we see that $\trace_{\widehat R}(\omega_{\widehat R})$ is a radical ideal of
$\widehat R$.
\end{proof}

Next we consider the CTR property under the quotient of an ideal generated by  a regular sequence.

\begin{prop}
Let $R$ be a local ring with canonical module or an $\NNN$-graded algebra over a field with
(irrelevant) maximal ideal $\mmmm$.
Suppose that $x_1$, \ldots, $x_r\in\mmmm$ is a (homogeneous) regular sequence with
$x_1$, \ldots, $x_r\in\trace_R(\omega_R)$.
Set $\overline R=R/(x_1, \ldots, x_r)R$.
Then $\trace_{\overline R}(\omega_{\overline R})=\trace_R(\omega_R)/(x_1,\ldots, x_r)R$.
In particular, $R$ is CTR if and only if so is $\overline R$.
\end{prop}
\begin{proof}
Localizing by $\mmmm$, we can reduce the $\NNN$-graded case to the local case.

First note that $\omega_{\overline R}=\omega_R/(x_1, \ldots, x_r)\omega_R=\omega_R\otimes_R \overline R$.
See \cite[Theorem 3.3.5]{bh}.
By \cite[Lemma 1.5 (ii)]{hhs}, we see that 
$$
\trace_R(\omega_R)\overline R\subset\trace_{\overline R}(\omega_R\otimes \overline R)
=\trace_{\overline R}(\omega_{\overline R}).
$$
On the other hand, let $M$ be an arbitrary maximal \cm\ $R$-module and $x'$ an $R$-regular element 
with $x'\in\trace_R(\omega_R)$.
Then by \cite[Theorem 2.3]{dkt}, we see that $x'\ext_R^1(M,R)=0$.
By considering the long exact sequence induced by
$0\to M\stackrel{x'}{\to}M\to M/x'M\to0$,
we get the following exact sequence.
$$
\ext_R^1(M,R)\stackrel{x'}{\to}\ext_R^1(M,R)\to\ext_R^2(M/x'M, R).
$$
Since $x'\ext_R^1(M,R)=0$ and $\ext_R^2(M/x'M,R)\cong\ext_{R/x'R}^1(M/x'M,R/x'R)$
(see, e.g.\ \cite[Lemma 3.1.16]{bh}), we see that
$\ann_R(\ext_R^1(M,R))\supset\ann_R(\ext_{R/x'R}^1(M/x'M, R/x'R))$.
Therefore, by \cite[Theorem 2.3]{dkt}, we see that
$
\trace_R(\omega_R)(R/x'R)\supset\trace_{R/x'R}(\omega_{R/x'R})
$.
Using this fact repeatedly, we see that
$$
\trace_{ R}(\omega_R)\overline R\supset \trace_{\overline R}(\omega_{\overline R}).
$$

Therefore,
$$
\trace_{\overline R}(\omega_{\overline R})=\trace_R(\omega_R)\overline R
=\trace_R(\omega_R)/(x_1, \ldots, x_r)R.
$$
In particular, $\trace_R(\omega_R)$ is a radical ideal if and only if so is $\trace_{\overline R}(\omega_{\overline R})$.
\end{proof}

Next, we consider the behavior of CTR property under the tensor product.

\begin{prop}
\mylabel{prop:tensor}
Let $R^{(1)}$ and $R^{(2)}$ be $\NNN$-graded $\KKK$-algebras, where $\KKK$ is a field and set
$R=R^{(1)}\otimes_\KKK R^{(2)}$.
\begin{enumerate}
\item
\mylabel{item:tensor factor ctr}
If $R$ is CTR and $\trace_{R^{(2)}}(\omega_{R^{(2)}})$ contains an $R^{(2)}$-regular element, then
$R^{(1)}$ is CTR.
\item
\mylabel{item:tensor ctr}
If $R^{(i)}$ is \cm\ and $(R^{(i)}/\trace_{R^{(i)}}(\omega_{R^{(i)}}))\otimes_\KKK R^{(3-i)}$ is reduced for $i=1, 2$, then
$R$ is CTR.
\end{enumerate}
\end{prop}
\begin{proof}
First note that by \cite[Proposition 4.1 and Theorem 4.2]{hhs}, it holds that
$\trace_R(\omega_R)=\trace_{R^{(1)}}(\omega_{R^{(1)}})R\cap\trace_{R^{(2)}}(\omega_{R^{(2)}})R$
and by \cite[Theorem 23.3 Corollary]{mat}, $R$ is \cm\ if and only if both $R^{(1)}$ and $R^{(2)}$ are \cm.

We first prove \ref{item:tensor factor ctr}.
Take an $R^{(2)}$-regular element $a$ from $\trace_{R^{(2)}}(\omega_{R^{(2)}})$.
Since
$$
0\to R^{(2)}\stackrel{a}{\to} R^{(2)}
$$
is exact, we see that
$$
0\to R^{(1)}/\trace_{R^{(1)}}(\omega_{R^{(1)}})\otimes R^{(2)}
\stackrel{1\otimes a}{\to} R^{(1)}/\trace_{R^{(1)}}(\omega_{R^{(1)}})\otimes R^{(2)}
$$
is also exact.
Since $R^{(1)}/\trace_{R^{(1)}}(\omega_{R^{(1)}})\otimes R^{(2)}=R/\trace_{R^{(1)}}(\omega_{R^{(1)}})R$,
we see that $1\otimes a\in R$ is an $R/\trace_{R^{(1)}}(\omega_{R^{(1)}})R$-regular element of $R$.
On the other hand, since $1\otimes a\in\trace_{R^{(2)}}(\omega_{R^{(2)}})R$, we see that
\begin{eqnarray*}
&&\trace_R(\omega_R)R[(1\otimes a)^{-1}]\cap R\\
&=&(\trace_{R^{(1)}}(\omega_{R^{(1)}})R\cap\trace_{R^{(2)}}(\omega_{R^{(2)}})R)R[(1\otimes a)^{-1}]\cap R\\
&=&\trace_{R^{(1)}}(\omega_{R^{(1)}})R.
\end{eqnarray*}
Thus, $\trace_{R^{(1)}}(\omega_{R^{(1)}})R$ is a radical ideal since $R$ is CTR.

Since $R^{(1)}\to R$, $x\mapsto x\otimes 1$ is  a faithfully flat homomorphism,
we see that $\trace_{R^{(1)}}(\omega_{R^{(1)}})=\trace_{R^{(1)}}(\omega_{R^{(1)}})R\cap R^{(1)}$
and therefore $\trace_{R^{(1)}}(\omega_{R^{(1)}})$ is a radical ideal of $R^{(1)}$.

Next, we prove \ref{item:tensor ctr}.
Since $\trace_R(\omega_R)=\trace_{R^{(1)}}(\omega_{R^{(1)}})R\cap
\trace_{R^{(2)}}(\omega_{R^{(2)}})R$,
it is enough to show that $R/\trace_{R^{(i)}}(\omega_{R^{(i)}})R$ is a reduced ring for $i=1, 2$.
However,
$$
R/\trace_{R^{(i)}}(\omega_{R^{(i)}})R=R^{(i)}/\trace_{R^{(i)}}(\omega_{R^{(i)}})\otimes R^{(3-i)}
$$
and the right hand side is assumed to be reduced, we see the result.
\end{proof}
\begin{remark}
\rm
In the situation of Proposition \ref{prop:tensor} \ref{item:tensor factor ctr}, if $R^{(2)}$ is reduced,
then for any $\pppp\in\ass(R^{(2)})$, $R^{(2)}_\pppp$ is a field and therefore \gor.
Thus $\trace_{R^{(2)}}(\omega_{R^{(2)}})\not\subset\pppp$.
Since $\pppp$ is an arbitrary associated prime of $R^{(2)}$, we see that $\trace_{R^{(2)}}(\omega_{R^{(2)}})$
contains an $R^{(2)}$-regular element.
On the other hand, if the assumption of Proposition \ref{prop:tensor} \ref{item:tensor ctr} is satisfied,
then $R^{(1)}$ and $R^{(2)}$ are reduced CTR rings.
Conversely, if $R^{(1)}$ and $R^{(2)}$ are reduced CTR and $\KKK$ is a perfect field, then
$R^{(i)}/\trace_{R^{(i)}}(\omega_{R^{(i)}})\otimes_\KKK R^{(3-i)}$ is reduced for $i=1, 2$ by
Lemma \ref{lem:tensor red} below.
\end{remark}
By the above remark, we see the following.

\begin{cor}
Let $\KKK$ be a field and let $R^{(1)}$ and $R^{(2)}$ be reduced $\NNN$-graded $\KKK$-algebras.
Then the followings hold.
\begin{enumerate}
\item
If $R^{(1)}\otimes R^{(2)}$ is CTR, then both $R^{(1)}$ and $R^{(2)}$ are CTR.
\item
If $\KKK$ is a perfect field and $R^{(1)}$ and $R^{(2)}$ are CTR, then $R^{(1)}\otimes R^{(2)}$
is CTR.
\end{enumerate}
\end{cor}

Next we consider the behavior of CTR property under the Segre product.
First, we state the following fact which is a direct consequence of \cite[Theorem 26.3]{mat}.

\begin{lemma}
\mylabel{lem:tensor red}
Let $\KKK$ be a perfect field and let $R^{(1)}$ and $R^{(2)}$ be reduced $\KKK$-algebras.
Then $R^{(1)}\otimes_\KKK R^{(2)}$ is a reduced ring.
\end{lemma}
Next we state the following.

\begin{lemma}
\mylabel{lem:segre rad}
Let $\KKK$ be a perfect field and let $R^{(1)}$, \ldots, $R^{(n)}$ be $\NNN$-graded reduced
$\KKK$-algebras and $I_i$ a graded radical ideal of $R^{(i)}$ for $1\leq i\leq n$.
Then $I_1\#\cdots \# I_n$ is a radical ideal of $R^{(1)}\#\cdots\# R^{(n)}$.
\end{lemma}
\begin{proof}
Since $R^{(1)}\#R^{(2)}$is a subring of $R^{(1)}\otimes_\KKK R^{(2)}$, we see that
$R^{(1)}\# R^{(2)}$ is a reduced ring by the previous lemma.
Therefore, by induction on $n$, it is enough to prove the case where $n=2$.

Since 
$$0\to I_1\to R^{(1)}\to R^{(1)}/I_1\to 0$$
is exact, we see that 
$$
0\to I_1\# R^{(2)}\to R^{(1)}\# R^{(2)}\to (R^{(1)}/I_1)\# R^{(2)}\to 0
$$
is exact.
Since $(R^{(1)}/I_1)\# R^{(2)}$ is a subring of $(R^{(1)}/I_1)\otimes R_2$ and $(R^{(1)}/I_1)\otimes R_2$
is reduced by Lemma \ref{lem:tensor red}, we see that $(R^{(1)}/I_1)\# R^{(2)}$ is a reduced ring.
Therefore, $I_1\#R^{(2)}$ is a radical ideal of $R^{(1)}\# R^{(2)}$.
We see by the same way that $R^{(1)}\#I_2$ is a radical ideal of $R^{(1)}\# R^{(2)}$.

Since $I_1\#I_2=(I_1\#R^{(2)})\cap(R^{(1)}\#I_2)$, we see that $I_1\#I_2$ is a radical
ideal of $R^{(1)}\# R^{(2)}$.
\end{proof}
Next we note a basic fact about \nzd\ and Segre product.

\begin{lemma}
\mylabel{lem:nzd}
Let $\KKK$ be a field, $R^{(1)}$ and $R^{(2)}$ $\NNN$-graded $\KKK$-algebras, 
$d$ a positive integer and $x_i\in R^{(i)}_d$ a \nzd\ of $R^{(i)}$ for $i=1,2$.
Then $x_1\# x_2$ is a \nzd\ of $R^{(1)}\# R^{(2)}$.
\end{lemma}
\begin{proof}
It is enough to show that for any non-zero homogeneous element $\alpha\in R^{(1)}\# R^{(2)}$, it holds
that $(x_1\#x_2)\alpha\neq0$.

Set $\deg\alpha=d'$, $\alpha=\sum_{i=1}^\ell z_i\# w_i$, $z_i\in R^{(1)}_{d'}$, $w_i\in R^{(2)}_{d'}$
$z_1$, \ldots, $z_\ell$ are linearly independent over $\KKK$ and $w_i\neq 0$ for $1\leq i\leq \ell$.
Then $\ell\geq1$ since $\alpha\neq 0$.
Further, $(x_1\#x_2)\alpha=\sum_{i=1}^\ell x_1z_i\# x_2w_i$,
$x_1z_1$, \ldots, $x_1z_\ell$ are linearly independent over $\KKK$ and $x_2w_i\neq 0$ for $1\leq i\leq \ell$.
Therefore, $(x_1\#x_2)\alpha\neq 0$.
\end{proof}
Now we show the following.

\begin{prop}
\mylabel{prop:segre}
Let $\KKK$ be a perfect field and let $R^{(1)}$, $R^{(2)}$, \ldots, $R^{(n)}$ be standard graded
$\KKK$-algebras.
Suppose that $R^{(i)}$ is a reduced CTR ring, $\dim R^{(i)}\geq 2$, $a(R^{(i)})<0$, 
$R^{(i)}$ contains a linear $R^{(i)}$-regular element and $R^{(i)}$ is level and anticanonical level for any $i$.
Set $a_i=a(R^{(i)})$ and $b_i=\min\{m:(\omega^{-1}_{R^{(i)}})_m\neq 0\}$ for $1\leq i\leq n$.
Suppose also that $a_i\geq a_1$ and $b_i\leq b_1$ for $2\leq i\leq n$.
Then the Segre product $R=R^{(1)}\#\cdots\#R^{(n)}$ is also a reduced CTR ring,
$a(R)=a_1$, $\min\{m : (\omega^{-1}_R)_m\neq 0\}=b_1$ and $R$ contains a linear $R$-regular element.
\end{prop}
\begin{proof}
By induction on $n$, it is enough to prove the case where $n=2$, since 
by \cite[Theorem 4.2.3]{gw}, 
$\dim(R^{(1)}\# R^{(2)})\geq 2$.
By \cite[Theorem 4.2.3]{gw}, 
$R^{(1)}\#R^{(2)}$ is \cm\ and
by \cite[Theorem 4.3.1]{gw} and \cite[Theorem 2.4]{hmp}, we see that
$\omega_R=\omega_{R^{(1)}}\#\omega_{R^{(2)}}$ and $\omega^{-1}_R=\omega^{-1}_{R^{(1)}}\#\omega^{-1}_{R^{(2)}}$.
Note the assumption of \cite[Theorem 2.4]{hmp} that $\KKK$ is an infinite field is used only for the existence
of linear \nzd.

Since $\omega_{R^{(1)}}$ is generated by $(\omega_{R^{(1)}})_{-a_1}$,
$\omega_{R^{(2)}}$ is generated by $(\omega_{R^{(2)}})_{-a_2}$ and $-a_2\leq -a_1$,
we see that $\omega_R$ is generated by $(\omega_{R^{(1)}})_{-a_1}\otimes (\omega_{R^{(2)}})_{-a_2}R^{(2)}_{a_2-a_1}$.
Similarly, we see that $\omega^{-1}_R$ is generated by 
$(\omega^{-1}_{R^{(1)}})_{b_1}\otimes(\omega^{-1}_{R^{(2)}})_{b_2}R_{b_1-b_2}$.
Therefore, $\trace_R(\omega_R)$ is generated by
$(\omega_{R^{(1)}})_{-a_1}(\omega^{-1}_{R^{(1)}})_{b_1}\otimes(\omega_{R^{(2)}})_{-a_2}(\omega^{-1}_{R^{(2)}})_{b_2}R_{a_2-a_1+b_1-b_2}$.
Since $b_1-a_1\geq b_2-a_2$, we see that 
\begin{eqnarray*}
\trace_R(\omega_R)
&=&\trace_{R^{(1)}}(\omega_{R^{(1)}})\#\trace_{R^{(2)}}(\omega_{R^{(2)}})_{\geq b_1-a_1}\\
&=&\trace_{R^{(1)}}(\omega_{R^{(1)}})\#\trace_{R^{(2)}}(\omega_{R^{(2)}}).
\end{eqnarray*}
Therefore, by Lemma \ref{lem:segre rad}, we see that $\trace_R(\omega_R)$ is a radical ideal,
by Lemma \ref{lem:nzd}, we see that there is a linear $R$-regular element and by
Lemma \ref{lem:tensor red}, we see that $R^{(1)}\otimes R^{(2)}$ is a reduced ring and therefore $R$,
a subring of $R^{(1)}\otimes R^{(2)}$, is also a reduced ring.
Further, since $\omega_R=\omega_{R^{(1)}}\#\omega_{R^{(2)}}$ is generated in degree $-a_1$ and 
$\omega^{-1}_{R^{(1)}}\#\omega^{-1}_{R^{(2)}}$ is generated in degree $b_1$, we see that
$a(R)=a_1$ and $\min\{m : (\omega^{-1})_m\neq 0\}=b_1$.
\end{proof}
In general, the converse of this proposition does not hold.
See Example \ref{ex:hibi}.

As a special case of Proposition \ref{prop:segre}, if $R^{(2)}$, \ldots, $R^{(n)}$ are \gor\ and
$a_1\leq a(R^{(i)})\leq b_1$ for $2\leq i\leq n$, then
$R^{(1)}\#R^{(2)}\#\cdots\#R^{(n)}$ is a CTR ring.
Note that in this case, we do not need to  assume that $R^{(1)}$ is a reduced ring,
since $\trace_R(\omega_R)=\trace_{R^{(1)}}(\omega_{R^{(1)}})\#R^{(2)}$ in the equation above.

We state an example of a pair of graded rings one of it is not CTR, but their Segre product is.
First we recall the definition of order polytope.
Let $P$ be a poset.
The convex polytope
$$
\conv\left\{f\in\RRR^P : 
\vcenter{\hsize=.5\textwidth\relax\noindent
$0\leq f(x)\leq 1$ for any $x\in P$, $x\leq y\Rightarrow f(x)\geq f(y)$}
\right\}
$$
is called the order polytope of $P$ and denoted by $\msOOO(P)$.
See \cite{sta}.
Note that we reverse the inequality of $f(x)$ and $f(y)$ above from \cite{sta} in order to make the
Ehrhart ring of $\msOOO(P)$ is identical with the Hibi ring $\RRRRR_\KKK[\msIII(P)]$ defined by Hibi \cite{hib},
where $\msIII(P)$ is the set of poset ideals of $P$.

For $n\in\ZZZ$, we set
$$
\TTTTTn(P)\define\left\{\nu\in\ZZZ^{P^-} :
\vcenter{\hsize=.5\textwidth\relax\noindent
$\nu(x)\geq n$ for any maximal element $x$ of $P$ and
if $x\covered y$ in $P^{-}$, then $\nu(x)\geq\nu(y)+n$}
\right\},
$$
where $x\covered y$ means that $y$ covers $x$, i.e.
$x<y$ and there is no $z\in P^-$ with $x<z<y$.
Then by \cite[Theorem 2.9]{mfiber}, it holds that
$$
\omega^{(n)}_\ekop=\bigoplus_{\nu\in\TTTTTn(P)}\KKK T^\nu
$$
for any $n\in\ZZZ$.

\begin{example}
\rm
\mylabel{ex:hibi}
Let $\KKK$ be a perfect field and let $P_1$, $P_2$ and $P_3$ be posets with the following Hasse diagrams.
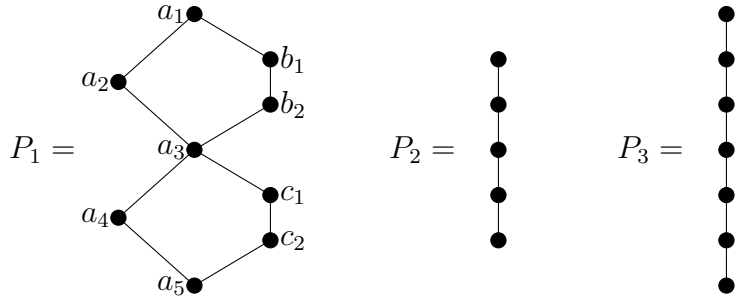
\begin{center}
\begin{tikzpicture}
\coordinate (a1) at (2,1.8);
\coordinate (a2) at (1,.9);
\coordinate (a3) at (2,0);
\coordinate (a4) at (1,-.9);
\coordinate (a5) at (2,-1.8);
\coordinate (b1) at (3,1.2);
\coordinate (b2) at (3,.6);
\coordinate (c1) at (3,-.6);
\coordinate (c2) at (3,-1.2);

\draw (a1)--(a2)--(a3)--(a4)--(a5);
\draw (a1)--(b1)--(b2)--(a3)--(c1)--(c2)--(a5);

\draw[fill] (a1) circle [radius=.1];
\draw[fill] (a2) circle [radius=.1];
\draw[fill] (a3) circle [radius=.1];
\draw[fill] (a4) circle [radius=.1];
\draw[fill] (a5) circle [radius=.1];
\draw[fill] (b1) circle [radius=.1];
\draw[fill] (b2) circle [radius=.1];
\draw[fill] (c1) circle [radius=.1];
\draw[fill] (c2) circle [radius=.1];

\node at (0,0) {$P_1=$};
\node at (1.7,1.8) {$a_1$};
\node at (.7,.9) {$a_2$};
\node at (1.7,0) {$a_3$};
\node at (.7,-.9) {$a_4$};
\node at (1.7,-1.8) {$a_5$};
\node at (3.3,1.2) {$b_1$};
\node at (3.3,.6) {$b_2$};
\node at (3.3,-.6) {$c_1$};
\node at (3.3,-1.2) {$c_2$};

\node at (5,0) {$P_2=$};
\coordinate (w1) at (6,1.2);
\coordinate (w2) at (6,.6);
\coordinate (w3) at (6,0);
\coordinate (w4) at (6,-.6);
\coordinate (w5) at (6,-1.2);

\draw (w1)--(w2)--(w3)--(w4)--(w5);

\draw[fill] (w1) circle [radius=.1];
\draw[fill] (w2) circle [radius=.1];
\draw[fill] (w3) circle [radius=.1];
\draw[fill] (w4) circle [radius=.1];
\draw[fill] (w5) circle [radius=.1];

\node at (8,0) {$P_3=$};
\coordinate (z1) at (9,1.8);
\coordinate (z2) at (9,1.2);
\coordinate (z3) at (9,.6);
\coordinate (z4) at (9,0);
\coordinate (z5) at (9,-.6);
\coordinate (z6) at (9,-1.2);
\coordinate (z7) at (9,-1.8);

\draw (z1)--(z7);

\draw[fill] (z1) circle [radius=.1];
\draw[fill] (z2) circle [radius=.1];
\draw[fill] (z3) circle [radius=.1];
\draw[fill] (z4) circle [radius=.1];
\draw[fill] (z5) circle [radius=.1];
\draw[fill] (z6) circle [radius=.1];
\draw[fill] (z7) circle [radius=.1];

\end{tikzpicture}
\end{center}
Then $E_\KKK[\msOOO(P_1)]$ is CTR.
In fact, let $T^\nu$ be an arbitrary monomial in $\sqrt{\trace(\omega_{E_\KKK[\msOOO(P_1)]})}$,
where $\nu\in\TTTTT^{(0)}(P_1)$.
Then by \cite[Theorem 4.5]{mp}, we see that $\nu(a_1)<\nu(a_3)<\nu(a_5)$.
 Set $I_1=\{b_i: \nu(b_i)>\nu(a_1)\}$, $I_2=\{c_i : \nu(c_i)>\nu(a_3)\}$
 and define $\zeta$ and $\eta\in\ZZZ^{P^-}$ by
 $$
 \zeta(x)=
 \left\{
 \begin{array}{ll}
 -i,&\quad x=a_i,\\
 -i-1+\chi_{I_1}(x),& \quad x=b_i,\\
 -i-3+\chi_{I_2}(x),& \quad x=c_i,\\
 -6,&\quad x=-\infty
 \end{array}
 \right.
 $$
 and $\eta=\nu-\zeta$.
 Then it is verified by hand calculation that $\zeta\in\TTTTT^{(-1)}(P_1)$ and $\eta\in\TTTTT^{(1)}(P_1)$.
 Since $\nu=\eta+\zeta$, we see that $T^\nu\in\trace(\omega_{E_\KKK[\msOOO(P_1)]})$.
 Thus, we see that $E_\KKK[\msOOO(P_1)]$ is a CTR ring.
 Further, we see 
 by \cite[Theorems 3.11 and 3.12]{mfiber} that $E_\KKK[\msOOO(P_1)]$ is level and anticanonical level,
 $a(E_\KKK[\msOOO(P_1)])=-8$ and 
 $\min\{m: (\omega^{(-1)}_{E_\KKK[\msOOO(P_1)]})_m\neq 0\}=-6$.

 Since $E_\KKK[\msOOO(P_2)]$ and $E_\KKK[\msOOO(P_3)]$ are isomorphic to
 polynomial rings with 6 and 8 variables respectively, we see that
 $E_\KKK[\msOOO(P_2)]$ and $E_\KKK[\msOOO(P_3)]$ are
 \gor\ rings with
 $a(E_\KKK[\msOOO(P_2)])=-6$ and $a(E_\KKK[\msOOO(P_3)])=-8$.
 Therefore, we see by Proposition \ref{prop:segre} that
 $$
 E_\KKK[\msOOO(P_1)]\#
 E_\KKK[\msOOO(P_2)]\#
 E_\KKK[\msOOO(P_3)]=
 E_\KKK[\msOOO(P_1\cup P_2\cup P_3)]
 $$
 is a CTR ring.
 However, by \cite[Theorem 2.7]{hmp}, the trace of the canonical module of
$E_\KKK[\msOOO(P_2)]\#
E_\KKK[\msOOO(P_3)]$
is the square of the irrelevant maximal ideal of
$E_\KKK[\msOOO(P_2)]\#
E_\KKK[\msOOO(P_3)]$.
Therefore, 
$E_\KKK[\msOOO(P_2)]\#
E_\KKK[\msOOO(P_3)]$
is not CTR.
\end{example}


\section{CTR property of Schubert cycles}

In this section and next, we state criteria of CTR property of certain classes of rings which motivated us
to define CTR property.
First in this section, we study Schubert cycles.

Before going into the details, we first establish notation and recall basic facts.
For the terms concerning algebras with straightening law (ASL for short) we consult \cite{bv}.
In particular, if $R$ is a graded ASL on a poset $\Pi$ over a field $\KKK$, $\Omega$ a poset ideal
of $\Pi$ and $I=\Omega R$,
we say that $\Omega$ or $I$ is straightening closed if for any incomparable elements
$\upsilon$, $\xi\in\Omega$,
every standard monomial $\mu_i$ appearing in the standard representation 
$$
\upsilon\xi=\sum_i c_i\mu_i, \quad c_i\in\KKK\setminus\{0\}
$$
has at least 2 factors in $\Omega$.
By \cite[Proposition 1]{dep}, we see the following.

\begin{lemma}
\mylabel{lem:str cl ideal}
Let $R$ be a graded ASL over $\KKK$ on a poset $\Pi$, $\Omega $ a straightening closed poset
ideal of $\Pi$ and $I=\Omega R$.
Then for any positive integer $n$, $I^n$ is an ideal of $R$ generated by
$\{\xi_1\cdots \xi_n : \xi_i\in\Omega$ for $1\leq i\leq n$, $\xi_1\leq\cdots\leq \xi_n\}$.
Also, $I^n$ is a $\KKK$-vector subspace of $R$ with basis 
$\{\xi_1\cdots\xi_\ell : \xi_i\in\Pi$ for $1\leq i\leq \ell$, $\xi_1\leq\cdots\leq\xi_\ell$, $\ell\geq n$, 
$\xi_1,\ldots, \xi_n\in\Omega\}$.
\end{lemma}
Since $\Omega$ is a poset ideal, $\xi_n\in\Omega$ implies $\xi_1$, \ldots, $\xi_n\in\Omega$.
However, we expressed the above lemma by the above form for convenience of later use.

For integers $m$ and $n$ with $1\leq m\leq n$ we set
$\Gamma(m\times n)\define\{[a_1,\ldots, a_m] : a_i\in\ZZZ$ for $1\leq i\leq m$, $1\leq a_1<\cdots<a_m\leq n\}$
and define the order on $\Gamma(m\times n)$ by
$$
[a_1, \ldots, a_m]\leq [b_1, \ldots, b_m]
\stackrel{\mathrm{def}}{\Longleftrightarrow} a_i\leq b_i \mbox{ for } 1\leq i\leq m.
$$
Then $\Gamma(m\times n)$ is a distributive lattice whose join, denoted by $\sqcup$,
and meet, denoted by $\sqcap$, are
\begin{eqnarray*}
&&[a_1, \ldots, a_m]\sqcup [b_1, \ldots, b_m]=[\max\{a_1,b_1\},\ldots,\max\{a_m,b_m\}]\\
&\mbox{ and }\\\relax
&&[a_1, \ldots, a_m]\sqcap [b_1, \ldots, b_m]=[\min\{a_1,b_1\},\ldots,\min\{a_m,b_m\}].
\end{eqnarray*}
Further, for $\gamma\in\Gamma(m\times n)$, we set
$\Gamma(m\times n;\gamma)\define\{\delta\in\Gamma(m\times n) : \delta\geq\gamma\}$.
Then $\Gamma(m\times n;\gamma)$ is a sublattice of $\Gamma(m\times n)$.

For an $m\times n$ matrix $M$ and $\gamma=[a_1, \ldots, a_m]\in\Gamma(m\times n)$,
we denote by $\gamma_M$ or $[a_1, \ldots, a_m]_M$ the $m$-minor of $M$ consisting of
columns $a_1$, \ldots, $a_m$.

Let $m$ and $n$ be integers with $1\leq m<n$, $V$ an $n$-dimensional $\KKK$-vector space and
$X=(X_{ij})$ an $m\times n$ matrix of indeterminates.
It is known that $\KKK$-subalgebra $G(X)$ of the polynomial ring
$\KKK[X_{ij} : 1\leq i\leq m$, $1\leq j\leq n]$ generated by the maximal minors of $X$ is the homogeneous
coordinate ring of the Grassmannian of the $m$-dimensional subspaces $\gr_m(V)$ of $V$.
Further $G(X)$ is an ASL on $\Gamma(m\times n)$ over $\KKK$ by the identification
$\Gamma(m\times n)\ni \gamma\leftrightarrow \gamma_X\in G(X)$.
See \cite[\S 4]{bv} for details.

We introduce the column degree, denoted by $\cdeg_j$ by setting
$$
\cdeg_j X_{k\ell}=
\left\{\begin{array}{ll}
1&\quad\ell=j,\\
0&\quad\ell\neq j
\end{array}
\right.$$ 
for $1\leq j\leq n$.
Further, we define grading of $G(X)$ by $\deg\gamma=1$ for any $\gamma\in\Gamma(m\times n)$.
Note that $\deg a=(1/m)\sum_{j=1}^n\cdeg_j a$ for any homogeneous element
$a\in G(X)$ in the $\NNN^n$-grading defined by column degree.
Note also for any incomparable elements $\upsilon$, $\xi\in\Gamma(m\times n)$ the standard
representation of $\upsilon\xi$ is of the following form.
$$
\upsilon\xi=\sum_{i} c_i\gamma_i\delta_i, \quad c_i\in\KKK, \gamma_i\leq \delta_i
$$
$$
\cdeg_j\upsilon+\cdeg_j\xi=\cdeg_j\gamma_i+\cdeg_j\delta_i
$$
for $1\leq j\leq n$ and for any $i$.
See \cite[\S 4]{bv} for details.

Now let $0=V_0\subset V_1\subset\cdots\subset V_n=V$ be a complete flag of $V$.
For integers $b_1$, \ldots, $b_m$ with $1\leq b_1<\cdots<b_m\leq n$, set
$\Omega(b_1, \ldots, b_m)\define\{W\in\gr_m(V) : \dim(W\cap V_{b_i})\geq i$ for $ 1\leq i\leq m\}$.
This is a subvariety of $\gr_m(V)$ called the Schubert subvariety of $\gr_m(V)$.
Set $a_i\define n-b_{m-i+1}+1$ for $1\leq i\leq m$ and $\gamma\define [a_1, \ldots, a_m]$.
Then the homogeneous coordinate ring of $\Omega(b_1, \ldots, b_m)$ is
$G(X)/(\delta\in\Gamma(m\times n) : \delta\not\geq \gamma)$
and called  the Schubert cycle.
By \cite[Proposition 1.2]{dep}, 
$G(X)/(\delta\in\Gamma(m\times n) : \delta\not\geq \gamma)$
is an ASL on $\Gamma(m\times n)\setminus\{\delta\in\Gamma(m\times n) : \delta\not\geq\gamma\}
=\Gamma(m\times n;\gamma)$.
Therefore the homogeneous coordinate ring of $\Omega(b_1, \ldots, b_m)$ is an ASL over 
$\KKK$ on $\Gamma(m\times n;\gamma)$.
We denote this ring by $G(X;\gamma)$.
See \cite[\S 1.D]{bv} for details.
Further $G(X;\gamma)$ is a normal domain by \cite[Theorem 6.3]{bv}.

Now we fix $\gamma=[a_1, \ldots, a_m]\in\Gamma(m\times n)$ and consider the canonical class
of $G(X;\gamma)$.
If $\gamma=[n-m+1, \ldots, n]$, then $G(X;\gamma)$ is 
isomorphic to 
a polynomial ring with 1 variable over $\KKK$.
Therefore, we assume that $\gamma \neq[n-m+1, \ldots, n]$ in the following.
We first decompose $\gamma$ into blocks and gaps as in \cite[\S 6]{bv}:
$[a_1, \ldots, a_m]=[\beta_0,\beta_1,\ldots,\beta_{t+1}]$,
$\beta_0=a_1, a_2, \ldots, a_{k(1)}$, $\beta_1=a_{k(1)+1}, a_{k(1)+2}, \ldots, a_{k(2)}$, \ldots,
$\beta_t=a_{k(t)+1}, a_{k(t)+2}, \ldots, a_{k(t+1)}$, $\beta_{t+1}=a_{k(t+1)+1}, a_{k(t+1)+2}, \ldots, n$.
$a_{j+1}-a_j=1$ if $k(i)<j<k(i+1)$ for some $i$ with $0\leq i\leq t+1$,
where $k(0)\define 0$,
$a_{k(i)+1}-a_{k(i)}\geq 2$ for $1\leq i\leq t+1$,
where $a_{m+1}\define n+1$.
Note that $\beta_{t+1}$ may be an empty block:
$\beta_{t+1}=\emptyset$ if and only if $a_m<n$.
This part is different from \cite{bv} but this makes the case dividing simpler.
We also define symbols of gaps between blocks by setting
$\chi_i\define\{j\in\ZZZ : a_{k(i+1)}<j<a_{k(i+1)+1}\}$ for $0\leq i\leq t$.

Next, we set
$\zeta_i\define[\beta_0, \beta_1, \ldots, \beta_{i-1}, a_{k(i)+1}, a_{k(i)+2}, \ldots, a_{k(i+1)-1}, a_{k(i+1)}+1, 
\beta_{i+1}, \ldots, \beta_{t+1}]$,
$\Omega_i\define\{\delta\in\Gamma(m\times n;\gamma): \delta\not\geq\zeta_i\}$
($=\Gamma(m\times n;\gamma)\setminus\Gamma(m\times n;\zeta_i)$)
and $J(x;\zeta_i)\define\Omega_i G(X;\gamma)$ for $0\leq i\leq t$.
Then $G(X;\gamma)/J(x;\zeta_i)\cong G(X;\zeta_i)$ is an integral domain with
$\dim(G(X;\gamma)/J(x;\zeta_i))=\dim G(X;\gamma)-1$.
In particular, $J(x;\zeta_i)$ is a height 1 prime ideal of $G(X;\gamma)$ for $0\leq i\leq t$.
Further,
$\gamma G(X;\gamma)=\bigcap_{i=0}^t J(x;\zeta_i)$
for $0\leq i\leq t$.
See \cite[\S\S 5 and 6]{bv}.
Note that $\Omega_i=\{[b_1,\ldots, b_m]\in\Gamma(m\times n;\gamma): b_{k(i+1)}=a_{k(i+1)}\}$
for $0\leq i\leq t$.

Set
$$
\kappa_i\define\sum_{j=0}^i |\beta_j|+\sum_{j=i}^t |\chi_j|
$$
for $0\leq i\leq t$,
$\kappa\define\max\{\kappa_i: 0\leq i\leq t\}$
and
$\kappa'\define\min\{\kappa_i: 0\leq i\leq t\}$.
Then by \cite[Theorem 8.12 and Corollary 8.13]{bv}, we see the following.

\begin{fact}
\mylabel{fact:bv 812}
The class
$$
\sum_{i=0}^t\kappa_i\cl(J(x;\zeta_i))
$$
in the divisor class group $\Cl(G(X;\gamma))$ is the canonical class of $G(X;\gamma)$
and $G(X;\gamma)$ is \gor\ if  and only if $\kappa-\kappa'=0$.
\end{fact}

Now we state the characterization of CTR property of $G(X;\gamma)$.

\begin{thm}
\mylabel{thm:schubert}
With the above notation, $G(X;\gamma)$ is CTR if and only if $\kappa-\kappa'\leq 1$.
\end{thm}
\begin{proof}
By Fact \ref{fact:bv 812}, $G(X;\gamma)$ is \gor\ if and only if $\kappa-\kappa'=0$.
Therefore, we may assume that $\kappa-\kappa'\geq 1$.

Set $J_i=J(x;\zeta_i)$ for $0\leq i\leq t$.
Then by \cite[Corollary 9.18]{bv}, we see that
$J_i^{(\ell)}=J_i^\ell$ for any positive integer $\ell$ and $0\leq i\leq t$.
Set also $\aaaa=\bigcap_{i=0}^t J_i^{\kappa_i}$.
Then by Fact \ref{fact:bv 812}, we see that $\aaaa$ is a canonical module of $G(X;\gamma)$ up to
shift of degree and $\trace(\aaaa)=\aaaa \aaaa^{(-1)}$ is the trace of the graded canonical module
of $G(X;\gamma)$.
(In fact, $\aaaa$ is the graded canonical module of $G(X;\gamma)$ by \cite[Proposition 3.7]{fm},
but we do not use this fact.)

First we consider the case where $\kappa-\kappa'=1$.
Set
\begin{eqnarray*}
I_1&=&\{i: 0\leq i\leq t, \kappa_i=\kappa\} \quad\mbox{and}\\
I_2&=&\{i: 0\leq i\leq t, \kappa_i=\kappa'\}.
\end{eqnarray*}
Then $I_1\cup I_2=\{0,1,\ldots, t\}$.
Since 
$\aaaa=\bigcap_{i=0}^t J_i^{\kappa_i}$ and $\Omega_i$ is straightening closed for any $i$
by \cite[Lemma 9.1]{bv},
$\aaaa$ is generated by
$\{\xi_1\cdots\xi_\kappa : \xi_1, \ldots, \xi_\kappa\in\Omega_i$ for $i\in I_1$ and
$\xi_1, \ldots, \xi_{\kappa-1}\in\Omega_i$ for $i\in I_2\}$
as an ideal of $G(X;\gamma)$ by Lemma \ref{lem:str cl ideal}.
Since $\bigcap_{i=0}^t\Omega_i=\{\gamma\}$, we see that
$\aaaa=\gamma^{\kappa-1}(\bigcap_{i\in I_1}J_i)$.
Note that $\bigcap_{i\in I_1}J_i$ is an ideal of $G(X;\gamma)$ generated by $\bigcap_{i\in I_1}\Omega_i$.

On the other hand, since $\aaaa^{(-1)}=\bigcap_{i=0}^t J_i^{(-\kappa_i)}$ and
$\gamma G(X;\gamma)=\bigcap_{i=0}^t J_i$,
we see that $\gamma^\kappa\aaaa^{(-1)}=\bigcap_{i=0}^t J_i^{(\kappa-\kappa_i)}=\bigcap_{i\in I_2}J_i$.
Therefore, $\gamma^\kappa\aaaa^{(-1)}$ is generated by $\bigcap_{i\in I_2}\Omega_i$ as an ideal of 
$\Gamma(m\times n;\gamma)$.
Therefore,
$$
\gamma^\kappa\trace(\aaaa)=\aaaa(\gamma^\kappa\aaaa^{(-1)})
=(\gamma^{\kappa-1}(\bigcap_{i\in I_1}J_i))(\bigcap_{i\in I_2} J_i)
$$
and we see that 
$$
\gamma\trace(\aaaa)=(\bigcap_{i\in I_1}J_i)(\bigcap_{i\in I_2} J_i).
$$
Thus, $\gamma\trace(\aaaa)$ is generated by $\{\xi\xi' : \xi\in\bigcap_{i\in I_1}\Omega_i, 
\xi'\in\bigcap_{i\in I_2}\Omega_i\}$.

Consider the standard representation of $\xi\xi'$ for arbitrary $\xi\in\bigcap_{i\in I_1}\Omega_i$ and
$\xi'\in\bigcap_{i\in I_2}\Omega_i$.
First note that $\xi\sqcap\xi'\in\bigcap_{i=0}^t\Omega_i=\{\gamma\}$ since $\Omega_i$ is a poset ideal
for any $i$.
Thus, $\xi\sqcap\xi'=\gamma$.
Suppose that
$$
\xi\xi'=\sum_{\ell} c_\ell\alpha_\ell\beta_\ell
$$
is the standard representation of $\xi\xi'$ in $G(X;\gamma)$.
Then, since $\xi\sqcap\xi'=\gamma\leq\alpha_\ell\leq\xi$, $\xi'$, 
$\alpha_\ell=\gamma$ for any $\ell$.
Moreover, since
$$
\cdeg_j\gamma+\cdeg_j\beta_\ell=\cdeg_j\xi+\cdeg_j\xi'=\cdeg_j\xi\sqcap\xi'+\cdeg_j\xi\sqcup\xi'
$$
for any $j$, we see that $\beta_\ell=\xi\sqcup\xi'$ for any $\ell$.
Therefore, the standard representation of $\xi\xi'$ is of the following form.
$$
\xi\xi'=c\gamma(\xi\sqcup\xi'), \quad c\in\KKK, c\neq0.
$$
(In fact, $c=1$, but we do not use this fact.)

Now consider $\xi\sqcup\xi'$.
Set $k(0)=0$ and
$$
\sigma_i\define[\beta_0, \ldots, \beta_{i-2}, a_{k(i-1)+1},  \ldots, a_{k(i)-1},
a_{k(i)+1}, \ldots, a_{k(i+1)}, a_{k(i+1)}+1, \beta_{i+1},  \ldots, \beta_{t+1}]
$$
for $1\leq i\leq t$ as in \cite[(6.8)]{bv},
$$
\Theta_i\define\Gamma(m\times n;\gamma)\setminus\Gamma(m\times n;\sigma_i)
$$
and $J(x;\sigma_i)=\Theta_i G(X;\gamma)$ for $1\leq i\leq t$.
Then $G(X;\gamma)/J(x;\sigma_i)\cong G(X;\sigma_i)$  is an integral domain and therefore $J(x;\sigma_i)$
is a prime ideal of $G(X;\gamma)$ for $1\leq i\leq t$.
Note that
$$
\Theta_i=\{[b_1, \ldots, b_m]\in\Gamma(m\times n;\gamma) : b_{k(i)}<a_{k(i)+1}\}
$$
for $1\leq i\leq t$.
Further, it is easily verified that $\zeta_{i-1}$, $\zeta_i\leq \sigma_i$ and therefore
$\Omega_{i-1}$, $\Omega_i\subset\Theta_i$ for $1\leq i\leq t$.

Set
$$
I'\define \{i : i\in I_1, i-1\in I_2\}\quad\mbox{and}\quad
I''\define\{i : i\in I_2, i-1\in I_1\}.
$$
Since $\xi\in\bigcap_{i\in I_1}\Omega_i$ and $\xi'\in\bigcap_{i\in I_2}\Omega_i$,
we see that $\xi\in\bigcap_{i\in I'\cup I''}\Theta_i$ and $\xi'\in\bigcap_{i\in I'\cup I''}\Theta_i$.
On the other hand, since 
$$
\Theta_i=\{[b_1, \ldots, b_m]\in\Gamma(m\times n;\gamma) : b_{k(i)}<a_{k(i)+1}\},
$$
we see that $\xi\sqcup\xi'\in\bigcap_{i\in I'\cup I''}\Theta_i$.
Thus, $\xi\sqcup\xi'\in \bigcap_{i\in I'\cup I''}J(x;\sigma_i)$.
Since $\xi$ (resp.\ $\xi'$) is an arbitrary element of $\bigcap_{i\in I_1}\Omega_i$
(resp.\ $\bigcap_{i\in I_2}\Omega_i$), we see that 
$$
\gamma\trace(\aaaa)\subset\gamma(\bigcap_{i\in I'\cup I''}J(x;\sigma_i)).
$$
Thus, $\trace(\aaaa)\subset\bigcap_{i\in I'\cup I''}J(x;\sigma_i)$.

Now consider the reverse inclusion.
It is enough to show that for any $\beta\in\bigcap_{i\in I'\cup I''}\Theta_i$, it holds that 
$\gamma\beta\in\gamma\trace(\aaaa)$.
Set $\beta=[b_1,\ldots, b_m]$.
Set also $k(0)=0$,
\begin{eqnarray*}
H_1&=&\{j\in\ZZZ: \exists i\in I_1; k(i)<j\leq k(i+1)\},\\ 
H_2&=&\{j\in\ZZZ: \exists i\in I_2; k(i)<j\leq k(i+1)\},\mbox{ and }\\
H_3&=&\{k(t+1)+1, \ldots, m\}.
\end{eqnarray*}
Note that $H_3=\emptyset $ if and only if $a_m<n$ and $a_j=b_j$ for $j\in H_3$.
We define integers $c_1$, \ldots, $c_m$ and $c'_1, \ldots, c'_m$ by
$$
c_j=
\left\{
\begin{array}{ll}
a_j,\quad \mbox{ if $j\in H_1\cup H_3$},\\
b_j,\quad \mbox{ if $j\in H_2$},
\end{array}
\right.
\qquad
c'_j=
\left\{
\begin{array}{ll}
a_j,\quad \mbox{ if $j\in H_2\cup H_3$},\\
b_j,\quad \mbox{ if $j\in H_1$}.
\end{array}
\right.
$$
Here we show the following key fact.

\begin{claim}
\mylabel{claim:in gamma}
It holds that $c_1<\cdots<c_m$ and $c'_1<\cdots<c'_m$.
\end{claim}
We prove the claim $c_1<\cdots<c_m$.
Claim $c'_1<\cdots<c'_m$ is proved similarly.

If $j$, $j+1\in H_1\cup H_3$, then $c_j=a_j<a_{j+1}=c_{j+1}$
and if $j$, $j+1\in H_2$, then $c_j=b_j<b_{j+1}=c_{j+1}$.
Assume that $j\in H_1\cup H_3$ and $j+1\in H_2$.
Then $c_j=a_j<a_{j+1}\leq b_{j+1}=c_{j+1}$.
Finally, assume that $j\in H_2$ and $j+1\in H_1\cup H_3$.
Then $j=k(i+1)$ for some $i\in I_2$.
If $i=t$, then $j+1\in H_3$ and we see that $c_j=b_j<b_{j+1}=c_{j+1}$.
If $i<t$, then $i\in I_2$ and $i+1\in I_1$.
Therefore, $i+1\in I'$.
Since $\beta\in\Theta_{i+1}$, we see that $c_j=b_{k(i+1)}<a_{k(i+1)+1}=c_{j+1}$
and the claim is proved.

By Claim \ref{claim:in gamma}, we see that $[c_1, \ldots, c_m]$, $[c'_1,\ldots, c'_m]\in\Gamma(m\times n)$.
Set $\xi\define[c_1, \ldots, c_m]$ and $\xi'\define[c'_1, \ldots, c'_m]$.
Then, since $\xi$, $\xi'\geq\gamma$, we see that $\xi$, $\xi'\in\Gamma(m\times n;\gamma)$.
Moreover, since $c_{k(i+1)}=a_{k(i+1)}$ for $i\in I_1$ (resp.\ $c'_{k(i+1)}=a_{k(i+1)}$ for $i\in I_2$),
we see that $\xi\in\bigcap_{i\in I_1}\Omega_i$ (resp.\ $\xi'\in\bigcap_{i\in I_2}\Omega_i$).
Since $\xi\sqcap \xi'=\gamma$ and $\xi\sqcup\xi'=\beta$, we see that the standard representation
of $\xi\xi'$ is the following form.
$$
\xi\xi'=c\gamma\beta,\quad c\in\KKK, c\neq0.
$$
Since $\xi\in \bigcap_{i\in I_1} J_i$ and $\xi'\in\bigcap_{i\in I_2} J_i$, we see that
$$
\gamma\beta=c^{-1}\xi\xi'\in(\bigcap_{i\in I_1}J_i)(\bigcap_{i\in I_2} J_i)=\gamma\trace(\aaaa).
$$
This is what we wanted to show and we see that
$
\trace(\aaaa)=\bigcap_{i\in I'\cup I''}J(x;\sigma_i)$.
Since $J(x;\sigma_i)$ is a prime ideal for any $i$, we see that $\trace(\aaaa)$ is a radical ideal.

Next consider the case where $\kappa-\kappa'\geq 2$.
Since $\aaaa=\bigcap_{i=0}^t J_i^{\kappa_i}$, we see that $\aaaa$ is generated by
$\{\xi_1\cdots\xi_\kappa : \xi_1, \ldots, \xi_{\kappa_i}\in\Omega_i$ for $0\leq i\leq t\}$.
On the other hand, since $\gamma^\kappa \aaaa^{(-1)}=\bigcap_{i=0}^t J_i^{\kappa-\kappa_i}$,
$\gamma^\kappa\aaaa^{(-1)}$ is generated by 
$\{\xi_1\cdots\xi_{\kappa-\kappa'} : \xi_1, \ldots, \xi_{\kappa-\kappa_i}\in\Omega_i$ for $0\leq i\leq t\}$.
Therefore $\gamma^\kappa\trace(\aaaa)=\aaaa(\gamma^\kappa \aaaa^{(-1)})
=(\bigcap_{i=0}^t J_i^{\kappa_i})(\bigcap_{i=0}^t J_i^{\kappa-\kappa_i})$ is generated by homogeneous elements
of degree $2\kappa-\kappa'$.
Thus $\trace(\aaaa)$ is generated by homogeneous elements of degree $\kappa-\kappa'$.
In particular, $\gamma\not\in\trace(\aaaa)$, since $\kappa-\kappa'\geq 2$.

On the other hand, by the above description, we see that
$\gamma^\kappa\in\bigcap_{i=0}^t J_i^{\kappa_i}$ and $\gamma^{\kappa-\kappa'}\in\bigcap_{i=0}^t J_i^{\kappa-\kappa'}$.
Thus, we see that 
$\gamma^{2\kappa-\kappa'}\in(\bigcap_{i=0}^t J_i^{\kappa_i})(\bigcap_{i=0}^t J_i^{\kappa-\kappa_i})
=\gamma^\kappa\trace(\aaaa)$,
and therefore
$\gamma^{\kappa-\kappa'}\in\trace(\aaaa)$.
Thus, we see that $\trace(\aaaa)$ is not a radical ideal.
\end{proof}

\begin{remark}
\rm
Ficarra et al. \cite[Theorem 1.1]{fhst} showed that if $\KKK$ is a field,
$m$ and $n$ are integers with $2\leq m\leq n$,
$X=(X_{ij})$ is an $m\times n$ matrix of indeterminates and
$t$ is an integer with $2\leq t\leq m$ then
$$
\trace_R(\omega_R)=I_{t-1}(X)^{n-m}R,
$$
where $R=\KKK[X_{ij} : 1\leq i\leq m, 1\leq j\leq n]/I_t(X)$.
Since $I_{t-1}(X)R$ is a prime ideal of $R$, we see that $R$ is CTR if and only if
$n-m\leq 1$.
While by \cite[Corollary 8.9]{bv} $R$ is \gor\ if and only if $n-m=0$.
\end{remark}


\section{CTR property of the Ehrhart rings of the stable set polytopes of  cycle graphs}

In this section, we establish a criterion of the CTR propety of the
Ehrhart ring of the stable set polytope of a cycle graph.
For basic terminology and facts of graph theory, we consult \cite{die}.

Let $G=(V,E)$ be a graph.
A stable set $S$ of $G$ is a subset of $V$ with no pair of elements in $S$ are adjacent.
$\emptyset$ and $\{v\}$ for any $v\in V$ are trivially stable.
We define the stable set polytope, denoted by $\stab(G)$ of $G$ by
$$
\stab(G)\define\conv\{\chi_S\in\RRR^V : S \mbox{ is a stable set of $G$}\}.
$$

Next we state the following.

\begin{definition}
\mylabel{def:xi+}
\rm
Let $X$ be a finite set and $\xi\in\RRR^X$.
For $B\subset X$, we set $\xip(B)\define\sum_{b\in B}\xi(b)$.
\end{definition}

We call a graph $G$ a cycle graph if $G$ consists of one cycle only, i.e.
$V=\{v_0, v_1, \ldots, v_{n-1}\}$, $E=\{\{v_i, v_j\} : i-j\equiv 1 \pmod n\}$
for some $n$ with $n\geq 3$.
A graph $G=(V,E)$ is called a t-perfect graph if
$$
\stab(G)=\left\{f\in\RRR^V : \vcenter{\hsize=.5\textwidth\relax\noindent
$0\leq f(v)\leq 1$ for any $v\in V$,
$f^+(e)\leq 1$ for any $e\in E$ and
$f^+(C)\leq\frac{|C|-1}{2}$ for any odd cycle $C$ without chord%
}
\right\}.
$$
It is known that a cycle graph is t-perfect. See \cite{mah}.

Let $G=(V,E)$ be a t-perfect graph.
For $n\in\ZZZ$, set
$$
\tUUUUUn\define\left\{\mu\in\ZZZ^{V^-} :
\vcenter{\hsize=.5\textwidth\relax\noindent
$\mu(v)\geq n$ for any $v\in V$,
$\mu^+(K)+n\leq\mu(-\infty)$ for any maximal clique $K$ in $G$ and
$\mu^+(C)+n\leq\frac{|C|-1}{2}\mu(-\infty)$ for any odd cycle $C$ without chord and length at least 5}
\right\}.
$$
Since $G$ is t-perfect, there are no cliques with size greater than 3, we see by \cite[Remark 3.10]{mhperfect} that
$$
\omega_{\eksg}^{(n)}=\bigoplus_{\mu\in\tUUUUUn}\KKK T^\mu
$$
for any $n\in \ZZZ$.

Let $G=(V,E)$ be a cycle graph.
If the length of the cycle is even, then $G$ is a bipartite graph and therefore is a perfect graph whose maximal
cliques have size 2.
Thus, by \cite[Theorem 2.1 (b)]{oh}, we see that $\eksg$ is a \gor\ ring.
Further, if the length $n$ of the cycle is odd, then by \cite[Theorem 1]{ht}, $\eksg$ is \gor\ if and only if
$n\leq 5$.

Therefore, we assume in the following of this section that $G=(V,E)$ is a cycle graph of odd length $n$
with $n\geq 7$.
Set $n=2\ell+1$, $V=\{v_0, v_1, \ldots, v_{2\ell}\}$ and we consider indices modulo $2\ell+1$,
$E=\{\{v_i,v_j\}: i-j\equiv 1 \pmod{2\ell+1}\}$, $e_i=\{v_i, v_{i+1}\}$ for $0\leq i\leq 2\ell$.
Further, we denote by $\omega$ the canonical ideal of $\eksg$.

Set 
$$
\pppp_i\define\bigoplus_{{\mu\in\tUUUUU^{(0)}}, \atop\mu(v_i)>0\  \mathrm{ or }\  \mu^+(V)<\ell\mu(-\infty)}
\KKK T^\mu
$$
for $0\leq i\leq 2\ell$.
Then $\pppp_i$ is a prime ideal of $\eksg$ and
$$
\sqrt{\trace(\omega)}=\bigcap_{i=0}^{2\ell}\pppp_i
$$
by \cite[Theorem 3.1]{mcycle}.

Now we state the following criterion of CTR property of $\eksg$.

\begin{thm}
\mylabel{thm:cycle graph}
Let $G=(V,E)$ be a cycle graph of odd length $n$.
Then $\eksg$ is a CTR ring if and only if $n\leq 7$.
\end{thm}
\begin{proof}
Set $n=2\ell+1$ as above.
First we prove the ``if'' part.
Suppose $\ell=3$, $\mu\in\tUUUUU^{(0)}$ and $T^\mu\in\bigcap_{i=0}^{2\ell}\pppp_i$.
Then by the proof of \cite[Lemmas 3.4 and 3.5]{mcycle}, we see that $T^\mu$ in $\trace(\omega)$.
Therefore, we see that $\bigcap_{i=0}^{2\ell}\pppp_i\subset\trace(\omega)$.
Since $\sqrt{\trace(\omega)}=\bigcap_{i=0}^{2\ell}\pppp_i$, we see that 
$$
\sqrt{\trace(\omega)}=\trace(\omega)
$$
and $\eksg$ is a CTR ring.

Next we prove the contraposition of ``only if'' part.
Suppose that $\ell\geq 4$ and define $\mu\in\ZZZ^{V^-}$ by
$$
\mu(x)=
\left\{
\begin{array}{ll}
1&\mbox{ if $x\in\{v_2,v_4, \ldots, v_{2\ell-2}, -\infty\}$},\\
0&\mbox{ otherwise.}
\end{array}
\right.
$$
Then, $\mu\in\tUUUUU^{(0)}$ and $\mu^+(V)=\ell-1<\ell=\ell\mu(-\infty)$.
Therefore, we see that $T^\mu\in\bigcap_{i=0}^{2\ell}\pppp_i=\sqrt{\trace(\omega)}$.
We assume that $\mu$ can be expressed as a sum of elements in $\tUUUUU^{(1)}$ and $\tUUUUU^{(-1)}$
and deduce a contradiction.

Suppose $\mu=\eta+\zeta$, $\eta\in\tUUUUU^{(1)}$ and $\zeta\in\tUUUUU^{(-1)}$.
Since $\eta(x)\geq 1$ for any $x\in V$ and $\eta^+(V)+1\leq \ell\eta(-\infty)$,
we see that
$$
\eta(-\infty)\geq\left\lceil\frac{2\ell+2}{\ell}\right\rceil=3.
$$
Similarly, since $\zeta(x)\geq -1$ for any $x\in V$ and $\zeta^+(V)-1\leq\ell\zeta(-\infty)$,
we see that
$$
\zeta(-\infty)\geq\left\lceil\frac{-2\ell-2}{\ell}\right\rceil=-2.
$$
On the other hand, since $\eta(-\infty)+\zeta(-\infty)=\mu(-\infty)=1$, we see that
$$
\eta(-\infty)=3\quad\mbox{ and }\quad\zeta(-\infty)=-2.
$$
Moreover, since $\eta(v_i)+\eta(v_{i+1})+1=\eta^+(e_i)+1\leq\eta(-\infty)=3$
and $\eta(v_j)\geq 1$ for any $i$ and $j$, we see that $\eta(v_j)=1$ for any $j$.
Thus,
$$
\zeta(x)=
\left\{
\begin{array}{ll}
0&\mbox{ if $x\in\{v_2, v_4, \cdots, v_{2\ell-2}\}$},\\
-2&\mbox{ if $x=-\infty$},\\
-1&\mbox{ otherwise.}
\end{array}
\right.
$$
Therefore, $\zeta^+(V)=-\ell-2$ and we see that
$$
\zeta^+(V)-1=-\ell-3>-2\ell=\ell\zeta(-\infty),
$$
since $\ell\geq 4$.
This contradicts to the assumption that $\zeta\in\tUUUUU^{(-1)}$.
Therefore, we see that $T^\mu\not\in\trace(\omega)$ and $\trace(\omega)$ is not a radical ideal.
\end{proof}


\section{A necessary condition for the Ehrhart ring of the stable set polytope of a 
perfect graph to be CTR}

In this section, we state a necessary condition of an Ehrhart ring of the stable set polytope of a
perfect graph to be CTR, 
which shows that CTR property is close to \gor\ property.
Let $G=(V,E)$ be a perfect graph and set
$k=\max\{|K| : K$ is a maximal clique of $G\}$ and
$k'=\min\{|K| : K$ is a maximal clique of $G\}$.

By Chv\'atal \cite[Theorem 3.1]{chv}, we see that
$$
\stab(G)=\left\{f\in\RRR^{V}: \vcenter{\hsize=.5\textwidth\relax\noindent
$f(x)\geq 0$ for any $x\in V$ and $f^+(K)\leq 1$ for any maximal clique $K$ of $G$}
\right\}.
$$
Further, by Ohsugi and Hibi \cite[Theorem 2.1 (b)]{oh}, $\eksg$ is \gor\ if and only if $k-k'=0$.
We recall notation and basic facts from \cite[Remark 3.8]{mhperfect}.
For $n\in\ZZZ$, we set
$$
\qUUUUUn\define\left\{
\mu\in\ZZZ^{V^-}:
\vcenter{\hsize=.5\textwidth\relax\noindent
$\mu(x)\geq n$ for any $x\in V$ and
$\mu^+(K)+n\leq \mu(-\infty)$ for any maximal clique $K$ of $G$}
\right\}.
$$
Then 
$$
\omega^{(n)}=\bigoplus_{\mu\in\qUUUUUn}\KKK T^\mu
$$
for $n\in\ZZZ$, where $\omega$ is the canonical ideal of $\eksg$.

Here, we state a very easily proved but very useful fact.

\begin{lemma}
\mylabel{lem:high school}
If $x\in V$, $\eta\in\qUUUUU^{(1)}$, $\zeta\in\qUUUUU^{(-1)}$ and $(\eta+\zeta)(x)=0$,
then $\eta(x)=1$ and $\zeta(x)=-1$.
\end{lemma}
Now we state the following necessary condition for $\eksg$ to be a CTR ring.

\begin{prop}
\mylabel{prop:perfect graph}
Let $G=(V,E)$ be a perfect graph.
If $\eksg$ is CTR, then $k-k'\leq 1$.
\end{prop}
\begin{proof}
We prove the contraposition of the proposition, so we assume that $k-k'\geq 2$.
Let $\mu$ be an element of $\ZZZ^{V^-}$ defined by
$$
\mu(x)=\left\{
\begin{array}{ll}
0&\mbox{ if $x\in V$},\\
1&\mbox{ if $x=-\infty$}.
\end{array}
\right.
$$
Then $\mu\in\qUUUUU^{(0)}$ and therefore $T^\mu\in\eksg$.
Further, if we define $\eta$, $\zeta\in\ZZZ^{V^-}$ by
$$
\eta(x)=\left\{
\begin{array}{ll}
1&\mbox{ if $x\in V$},\\
k+1&\mbox{ if $x=-\infty$},
\end{array}
\right.
\quad\mbox{ and }\quad
\zeta(x)=\left\{
\begin{array}{ll}
-1&\mbox{ if $x\in V$},\\
-k'-1&\mbox{ if $x=-\infty$},
\end{array}
\right.
$$
then $\eta\in\qUUUUU^{(1)}$, $\zeta\in\qUUUUU^{(-1)}$ and $(k-k')\mu=\eta+\zeta$.
Therefore,
$$
(T^\mu)^{k-k'}\in\trace(\omega).
$$

On the other hand, if there are $\eta'\in\qUUUUU^{(1)}$ and $\zeta'\in\qUUUUU^{(-1)}$ with
$\mu=\eta'+\zeta'$, then $\eta'(x)=1$ and $\zeta'(x)=-1$ for any $x\in V$ by Lemma \ref{lem:high school}.
Therefore, $\eta'(-\infty)\geq k+1$ and $\zeta'(-\infty)\geq -k'-1$.
Thus,
$$
1=\mu(-\infty)=\eta'(-\infty)+\zeta'(-\infty)\geq k-k'\geq 2.
$$
This is a contradiction.
Therefore, $T^\mu\not\in\trace(\omega)$ and we see that $\trace(\omega)$ is not a radical ideal.
\end{proof}

As the following example shows, the condition in the above proposition is not sufficient.

\begin{example}
\rm
Let $G=(V,E)$ be the following graph.
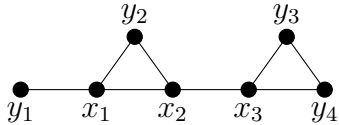
\begin{center}
\begin{tikzpicture}
\coordinate (Y1) at (0,0);
\coordinate (X1) at (1,0);
\coordinate (X2) at (2,0);
\coordinate (X3) at (3,0);
\coordinate (Y4) at (4,0);
\coordinate (Y2) at (1.5,.7);
\coordinate (Y3) at (3.5,.7);

\draw (Y1)--(X1)--(X2)--(X3)--(Y4);
\draw (X1)--(Y2)--(X2);
\draw (X3)--(Y3)--(Y4);

\draw[fill] (X1) circle [radius=0.1];
\draw[fill] (X2) circle [radius=0.1];
\draw[fill] (X3) circle [radius=0.1];
\draw[fill] (Y1) circle [radius=0.1];
\draw[fill] (Y2) circle [radius=0.1];
\draw[fill] (Y3) circle [radius=0.1];
\draw[fill] (Y4) circle [radius=0.1];

\node at (0,-.3) {$y_1$};
\node at (1,-.3) {$x_1$};
\node at (2,-.3) {$x_2$};
\node at (3,-.3) {$x_3$};
\node at (4,-.3) {$y_4$};
\node at (1.5,1) {$y_2$};
\node at (3.5,1) {$y_3$};

\end{tikzpicture}
\end{center}
This is a comparability graph of a poset $P$ whose Hasse diagram is

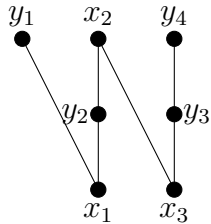
\begin{center}
\begin{tikzpicture}
\coordinate (y1) at (0,2);
\coordinate (x1) at (1,0);
\coordinate (y2) at (1,1);
\coordinate (x2) at (1,2);
\coordinate (x3) at (2,0);
\coordinate (y3) at (2,1);
\coordinate (y4) at (2,2);

\draw (y1)--(x1)--(x2)--(x3)--(y4);

\draw[fill] (y1) circle [radius=.1];
\draw[fill] (y2) circle [radius=.1];
\draw[fill] (y3) circle [radius=.1];
\draw[fill] (y4) circle [radius=.1];
\draw[fill] (x1) circle [radius=.1];
\draw[fill] (x2) circle [radius=.1];
\draw[fill] (x3) circle [radius=.1];

\node at (0,2.3) {$y_1$};
\node at (1,-.3) {$x_1$};
\node at (.7,1) {$y_2$};
\node at (1,2.3) {$x_2$};
\node at (2,-.3) {$x_3$};
\node at (2.3,1) {$y_3$};
\node at (2,2.3) {$y_4$};

\end{tikzpicture}
\end{center}
i.e.
$G=(V,E)$, $V=P$, $E=\{\{z,w\}: z,w\in P, z<w\}$.
In particular, $G$ is perfect.
Further, 
$\max\{|K|: K$ is a maximal clique in $G\}=3$ and
$\min\{|K|: K$ is a maximal clique in $G\}=2$.

Define $\mu\in\ZZZ^{V^-}$ by
$$
\mu(z)=\left\{
\begin{array}{ll}
1&\mbox{ if $z\in\{x_1, x_2, x_3\}$},\\
0&\mbox{ if $z\in\{y_1,y_2,y_3,y_4\}$},\\
2&\mbox{ if $z=-\infty$}.
\end{array}
\right.
$$
Then $\mu\in\qUUUUU^{(0)}$.

We first show that $T^\mu\not\in\trace(\omega)$.
Assume the contrary.
Then there are $\eta\in\qUUUUU^{(1)}$ and $\zeta\in\qUUUUU^{(-1)}$ with $\mu=\eta+\zeta$.
Then $\eta(y_i)=1$ and $\zeta(y_i)=-1$ for $1\leq i\leq 4$ by Lemma \ref{lem:high school}.
Further,
\begin{eqnarray*}
\eta(-\infty)+\zeta(-\infty)&=&\mu(-\infty)=2,\\
\eta(x_i)+\zeta(x_i)&=&\mu(x_i)=1\quad\mbox{ for $1\leq i\leq 3$},\\
\zeta(x_1)+\zeta(y_1)-1&\leq&\zeta(-\infty),\\
\eta(x_1)+\eta(x_2)+\eta(y_2)+1&\leq&\eta(-\infty),\\
\zeta(x_2)+\zeta(x_3)-1&\leq&\zeta(-\infty)\\
\noalign {and} \\
\eta(x_3)+\eta(y_3)+\eta(y_4)+1&\leq&\eta(-\infty).
\end{eqnarray*}
Therefore 
\begin{eqnarray*}
4&=&2\mu(-\infty)\\
&=&2\eta(-\infty)+2\zeta(-\infty)\\
&\geq&(\zeta(x_1)-2)+(\eta(x_1)+\eta(x_2)+2)+(\zeta(x_2)+\zeta(x_3)-1)+(\eta(x_3)+3)\\
&=&\eta(x_1)+\eta(x_2)+\eta(x_3)+\zeta(x_1)+\zeta(x_2)+\zeta(x_3)+2\\
&=&\mu(x_1)+\mu(x_2)+\mu(x_3)+2\\
&=&5.
\end{eqnarray*}
This is a contradiction.
Thus, we see that $T^\mu\not\in\trace(\omega)$.

Next, we show that $T^\mu\in\sqrt{\trace(\omega)}$.
Since $G$ is a comparability graph of $P$, $\stab(G)=\msCCC(P)$ by \cite[Theorem 3.1]{chv},
where $\msCCC(P)$ is the chain polytope of $P$.
See \cite{sta} for the definition of the chain polytope.
Therefore, by \cite[Theorem 3.7]{mp}, we see that $T^\mu\in\sqrt{\trace(\omega)}$.
In fact, by using the idea of the proof of \cite[Theorem 3.7]{mp}, we can construct $\eta\in\qUUUUU^{(1)}$ and
$\zeta\in\qUUUUU^{(-1)}$ as follows.
\begin{eqnarray*}
\eta(z)&=&
\left\{
\begin{array}{ll}
1&\mbox{ if $z\in\{y_1,y_2,y_3,y_4\}$},\\
2&\mbox{ if $z\in\{x_1,x_2\}$},\\
3&\mbox{ if $z=x_3$},\\
6&\mbox{ if $z=-\infty$},
\end{array}
\right.\\
\zeta(z)&=&
\left\{
\begin{array}{ll}
0&\mbox{ if $z\in\{x_1,x_2\}$},\\
-1&\mbox{ if $z\in\{x_3,y_1,y_2,y_3,y_4\}$},\\
-2&\mbox{ if $z=-\infty$}
\end{array}
\right.
\end{eqnarray*}
Then $\eta+\zeta=2\mu$ and therefore $(T^\mu)^2\in\trace(\omega)$.
\end{example}

We end with the following.

\begin{prob}
\rm
Give a criterion of the Ehrhart ring of the chain polytope of a poset to be CTR.
\end{prob}





%

\end{document}